\documentclass[12pt]{amsart}
\usepackage{amssymb,latexsym}
\usepackage{enumerate}

\makeatletter
\@namedef{subjclassname@2010}{%
  \textup{2010} Mathematics Subject Classification}
\makeatother

\theoremstyle{definition}

\numberwithin{equation}{section}

\newtheorem{theorem}{Theorem}[section]
\newtheorem{lemma}[theorem]{Lemma}

\newtheorem{corollary}[theorem]{Corollary}
\theoremstyle{definition}
\newtheorem{definition}[theorem]{Definition}

\newtheorem{remark}[theorem]{Remark}
\newtheorem{example}[theorem]{Example}
\newtheorem{examples}[theorem]{Examples}

\begin{document}


\baselineskip=17pt




\newcommand{\lcm}{\operatorname{lcm}}
\def\C{{\mathbb C}}
\def\Z{{\mathbb Z}}
\def\Ok{{\mathcal{O}_K}}
\def\Q{{\mathbb Q}}
\def\RR{\mathfrak{R}}
\def\AP{\mathop{\sf AP}}
\def\Mult{{\mathsf M}}
\makeatletter
\def\imod#1{\allowbreak\mkern10mu({\operator@font mod}\,\,#1)}
\def\imod#1{\allowbreak\mkern10mu({\operator@font mod}\,\,#1)}
\makeatother

\title{Arithmetic Progressions in a Unique Factorization Domain}

\author{Sudhir R. Ghorpade}
\address{Department of Mathematics, 
Indian Institute of Technology Bombay, 
Powai, Mumbai 400076, India.}
\email{srg@math.iitb.ac.in}

\author{Samrith Ram}
\address{Department of Mathematics,
Indian Institute of Technology Bombay, 
Powai, Mumbai 400076, India.}
\email{samrith@gmail.com}

\date{\today}

\begin{abstract}
Pillai showed that any sequence of consecutive integers with at most 16 terms 
possesses one term that is relatively prime to all the others. We give a new proof of a slight generalization of this result to arithmetic progressions of integers and further extend it 
to arithmetic progressions in unique factorization domains of characteristic zero. 
\end{abstract}



\maketitle

\section{Introduction}
In an attempt to prove a conjecture that products of consecutive integers are never perfect powers, Pillai \cite{P1} (see also \cite{BT}), in the early 1940's, considered the problem of finding sets of positive integers having the  property that they possess an element that is relatively prime to all the rest. He showed that in any set of at most 16 consecutive integers there exists one that is relatively prime to the others. In addition, he proved that for $17 \leq m \leq 430$, there exist infinitely many sets of $m$ consecutive integers which possess no element that is relatively prime to all the rest. Pillai, in fact, believed that the latter result is true for all $m \geq 17$, and this was soon confirmed by Brauer \cite{Br} and independently by Pillai himself \cite{P3,P4} using a result of Erd{\H{o}}s \cite[Theorem II]{Erdos}; for more recent proofs, 
one may refer to Evans \cite{Ev1}, Harborth \cite{Har}, and Eggleton \cite{Egg}.
In what follows, we shall refer to the former result as the Pillai Theorem and the latter as the Brauer-Pillai Theorem. 

It is not difficult to show that the Pillai Theorem is applicable not only to sets of at most 16 consecutive integers, but also to arithmetic progressions of integers with at most 16 terms, and we will refer to this fact 
as the Generalized Pillai Theorem. 
While we have not found in the existing literature a proof of this  
fact (see, however, Remark \ref{ConsecAndAP}), 
analogues of the Brauer-Pillai Theorem for arithmetic progressions are 
readily found, and we cite, in particular, the works of 
Evans \cite{Ev}, Ohtomo and Tamari \cite{Ohto}, and of Hajdu and Saradha \cite{HajSar}. 
Moreover, numerous extensions, analogues and generalizations of the Brauer-Pillai Theorem have been considered in the recent past; see, for example,
Caro \cite{Caro}, Gassko \cite{Gas}, Saradha and Thangadurai \cite{Sar}, and also the works cited above.
It may be remarked that all these extensions are 
in the setting of integers and use techniques from elementary 
or analytic number theory.  

We consider in this paper an extension of the Generalized Pillai Theorem in a 
wider algebraic context. Thus, we ask, if a similar result holds for Gaussian integers, or more generally, for rings of integers of algebraic number fields of class number one, or even more generally, for arbitrary integral domains where the notion of GCD (and hence of two elements being relatively prime) makes sense. 
Our main result is an analogue of the Generalized Pillai Theorem for the
so-called $\sigma$-atomic GCD domains of characteristic zero, and in particular, for arbitrary unique factorization domains of characteristic zero. This is achieved partly by introducing 
an invariant associated to an integral domain,  called its \emph{decomposition number}.
It is then proved that if $R$ is a UFD of characteristic zero with decomposition number $\delta_R$ and if $N:=\min\{16, 1+{\delta}_R\}$, then any arithmetic progression of at most $N$ terms 
with the first term coprime to the common difference contains a term that is relatively prime to all the rest. 
It is also shown that
$N:=\min\{16, 1+ \delta_R\}$ is the maximum possible number with this property.
As a special case, one sees that the Generalized Pillai Theorem holds for the Gaussian integers with $16$ replaced by~$6$. Our proof of the general result makes use of the corresponding result 
for integers. With this in view, and in a bid to make this paper self-contained, we include in Section \ref{Pillai} below a fairly short proof of the Generalized Pillai Theorem for arithmetic progressions of integers, which gives, in particular, a new proof of the Pillai Theorem. Some ring theoretic preliminaries and the notion of decomposition number are discussed in Section \ref {general}. The main result is proved in Section~\ref{last}.

\section{Arithmetic Progressions of Integers}
\label{Pillai}
Let us begin with some notations and terminology, which will be used in the remainder of this paper. 
Let $R$ be an integral domain. For $r \in R$ and $S \subseteq R$, we denote  by $\Mult(r,S)$ the set $\{s\in S:r \mid s\}$ of all multiples of 
$r$ in $S$. 
For $a,d \in R$ and a positive integer $n$, we denote by $\AP(a,d,n)$ the set $\{a,a+d,\ldots,a+(n-1)d\}$ of elements of the arithmetic progression with $n$ terms having $a$  as its first term and $d$ the common difference. Further, if $R$ is a GCD domain (i.e. an integral domain in which any two elements have a greatest common divisor) and $m,n$ are positive integers, then we shall write $\{a_1,a_2,\ldots,a_m\} \perp \{b_1,b_2,\ldots,b_n\}$ to mean that $\gcd(a_i,b_j )=1$ for $1\leq i\leq m$ and $1\leq j \leq n$. Here, as usual, $\gcd(a,b)$ denotes a greatest common divisor of $a,b\in R$, and although it is determined only up to multiplication by a unit, statements such as ``$\gcd(a,b)=1$'' or ``$\gcd(a,b)$ divides $c$'' 
have an unambiguous and obvious meaning, and we shall continue to use them. Of course if $R=\Z$ is the ring of integers, then $\gcd(a,b)$ is  unique since we require it to be positive if $a,b\in \Z$ are not both zero and set $\gcd(0,0):=0$. Also if $R=\Z$ and $n$ is a positive integer, then $\le$ will denote the componentwise partial order on $\Z^n$ so that for any \mbox{$a_1, \dots , a_n, b_1, \dots , b_n\in \Z$}, 
$(a_1,\ldots,a_n)\leq (b_1,\ldots,b_n) \Longleftrightarrow a_i\leq b_i$ for all $i=1,\ldots,n$. Finally, for a finite set $A$, we denote by $|A|$ the cardinality of $A$. 
  
\begin{theorem}[Generalized Pillai Theorem] 
\label{progressions}
Let $a,d$ be coprime integers and let $n$ be a positive integer $ \leq 16$. Then the arithmetic progression $a,a+d,\ldots,a+(n-1)d$ contains a term that is relatively prime to all the others.
\end{theorem}

We first make an elementary observation and record a useful consequence thereof. 

\begin{lemma}
\label{le1}
If $R$ is a GCD domain, $a,d \in R$ are coprime and $r,s$ are nonnegative integers, then $$\gcd(a+rd,a+sd)\mid (r-s).$$
\end{lemma}
\begin{proof}
Any common divisor of $a+rd$ and $a+sd$ divides both $(a+rd)-(a+sd)$ and $s(a+rd)-r(a+sd)$. Since $\gcd(a,d)=1$, the lemma follows.
\end{proof}	

\begin{corollary}
\label{mult}
If $a,d$ are coprime integers and $m$ is a positive integer, then 
$$|\Mult(m,\AP(a,d,n))|\leq \lceil n/m \rceil$$ 
where $\lceil x \rceil$ denotes the least integer $\geq x$. 
\end{corollary}
     
We now proceed to prove Theorem \ref{progressions}. For $n=1$, the theorem is vacuously true. Let us first consider the case in which all terms of $\AP(a,d,n)$ are odd. The case where $n=2$ is trivial. If $n=3,4$ or 5, then the term $a+2d$ is relatively prime to all the others. If $n=6$, then 
one of $a+2d,a+3d$ is not divisible by 3 and is relatively prime to all the other terms. If $n=7,8,9,10$ or 11, one of $a+4d,\, a+5d,\, a+6d$ is coprime to both 3 and 5 and consequently is relatively prime to all the other terms. If $12\leq n\leq 16$, some element in the set 
$\{a+id:6\leq i\leq 10\}$ is coprime to 3, 5 and 7 and is relatively prime to all the others in $\AP(a,d,n)$. Thus Theorem~\ref{progressions} holds whenever all terms of $\AP(a,d,n)$ are odd, or equivalently, when $d$ is even.

It remains to consider the case when terms in the progression are alternately even and odd. The case where $n=2$ is trivial. If $n=3$, the term $a+d$ is relatively prime to the others. If $n=4$ or 5, at least two terms in the progression are odd and one of those is not divisible by 3. This number is relatively prime to the others. 
To settle the remaining cases, the following two lemmas will be useful.

\begin{lemma}
\label{le2}
Let $k$ be an integer with $3\le k\le 8$. 
Suppose Theorem~\ref{progressions} holds for any coprime integers $a,d$ with $d$ odd and for $n=2k-1$. 
Then it also holds for any coprime integers $a,d$ with $d$ odd and for $n=2k$.
\end{lemma}
\begin{proof} 
Let $a,d$ be coprime integers with $d$ odd.  
Write $a_i:=a+(i-1)d$ for $i=1, \dots , 2k$ and $A:=\AP(a,d,2k)=\{a_1,\ldots,a_{2k}\}$. 
Assume first that $a_1$ is odd. Now $A\setminus \{a_1\}=\AP(a_2,d,2k-1)$ and hence there exists $m\in \Z$ with $2\le m \le 2k$ 
such that  $a_m\perp A\setminus\{a_1,a_m\}$. This implies that $a_m\perp \{1,\ldots,m-2\}$. Moreover, $m$ is odd since both $a_1, d$ are odd and $k\ge 3$. Consequently,  $\gcd(a_1,a_m )\mid (m-1)/2$. On the other hand, $1\le (m-1)/2\leq m-2$, since $m\geq 3$. It follows that \hbox{$\gcd(a_m,(m-1)/2)=1$} and therefore $\gcd(a_1,a_m )=1$. As a result, ${a_m }\perp A\setminus\{a_m\}$. In case $a_1$ is even, $a_{2k}$ is odd and we use the same argument for $\AP(a_{2k},-d,2k)$.
\end{proof}
		    
\begin{lemma}
\label{le3} 
Let $k$ be an integer with $3\le k\le 7$. 
Suppose Theorem~\ref{progressions} holds for any odd coprime integers $a,d$ and for $n=2k$. 
Then it holds for any odd coprime integers $a,d$ and for $n=2k+1$.
\end{lemma}
\begin{proof} 
The proof is similar to that of Lemma~\ref{le2}. Let $a,d$ be odd coprime integers. Write $a_i=a+(i-1)d$ for $i=1, \dots , 2k+1$ and $A=\AP(a,d,2k+1)$. By the hypothesis, there exists 
$m\in \Z$ with $2\le m \le 2k+1$ such that $a_m\perp A\setminus\{a_1,a_m\}$. This implies that $a_m\perp \{1,\ldots,m-2\}$. Moreover, $m$ is odd, $\gcd(a_1,a_m )\mid(m-1)/2$, and $(m-1)/2\leq m-2$, since $m\geq 3$. Thus $\gcd(a_m, a_1)=1$ 
and $a_m\perp A\setminus\{a_m\}$.
\end{proof}
    			    
In view of the two lemmas above and the discussion preceding it, 
it suffices to prove the theorem for coprime integers $a,d$ with $a$ even, $d$ odd, and for odd integers $n$ with $7\le n < 16$. 
Fix such $a,d,n$ and let $a_i:=a+(i-1)d$ for $i=1, \dots , n$ and $A:= \AP(a,d,n)=\{a_1,\ldots,a_n\}$. We now proceed by a case-by-case argument. 

First, suppose $n=7$. Let $B:=\{a_2,a_4,a_6\}=\AP(a_2,2d,3)$. By Corollary~\ref{mult}, $|\Mult(3,B)\cup \Mult(5,B)|\leq 2<|B|$. Hence it follows that there exists $x\in B$ such that $x\perp\{2,3,5\}$. Consequently, $x\perp A\setminus\{x\}$.

Next, suppose $n=9$. Let $B:=A\setminus \Mult(2,A)=\AP(a_2,2d,4)$. By Corollary~\ref{mult}, 
$$
\left(|\Mult(3,B)|,|\Mult(5,B)|,|\Mult(7,B)|\right)\leq (2,1,1).
$$ 
Since $|B|=4$, if there is a strict inequality in one of the coordinates, then there is $x\in B$ such that $x\perp \{2,3,5,7\}$ and consequently $x\perp A\setminus\{x\}$. If equality holds in all the coordinates, then $|\Mult(3,B)|=2$ and we must necessarily have $\Mult(3,B)=\{a_2,a_8\}$. 
But then $\{a_4,a_6\}\perp\{2,3\}$ and the one among $a_4$ and $a_6$ that is coprime to 5 is relatively prime to all other elements in $A$.
 
For $n=11$, let $B:=A\setminus \Mult(2,A)=\AP(a_2,2d,5)$. 
By Corollary~\ref{mult},
$$
\left(|\Mult(3,B)|,|\Mult(5,B)|,|\Mult(7,B)|\right)\leq (2,1,1).
$$ 
Since $2+1+1<5=|B|$, there exists $x\in B$ with $x\perp \{2,3,5,7\}$ 
and so $x\perp A\setminus\{x\}$.
 
We now consider the case $n=13$. Let $B:=A\setminus \Mult(2,A)=\AP(a_2,2d,6)$. Then $|B|=6$ and by Corollary~\ref{mult}, 
$$
\left(|\Mult(3,B)|,|\Mult(5,B)|,|\Mult(7,B)|,|\Mult(11,B)|\right)\leq (2,2,1,1).
$$
If there is a strict inequality in one of the coordinates we are through. So suppose equality holds in all coordinates. This forces $\Mult(5,B)=\{a_2,a_{12}\}$. So if we let $B_1:=\{a_4,a_6,a_8\}$, then $B_1 \perp\{2,5\}$ and by Corollary~\ref{mult}, 
$(|\Mult(3,B_1) |,|\Mult(7,B_1)|)\leq(1,1)$. Thus there exists $x\in B_1$ such that $x\perp\{2,3,5,7\}$ and hence $x\perp A\setminus\{x\}$.

Finally, suppose $n=15$. Let $B:=A\setminus (\Mult(2,A)\cup \{a_2,a_{14}\})=\AP(a_4,2d,5)$. Then $|B|=5$ and by Corollary~\ref{mult}, 
$$
\left(|\Mult(3,B)|,|\Mult(5,B)|,|\Mult(7,B)|,|\Mult(11,B)|\right)\leq (2,1,1,1).
$$
If there is a strict inequality in one of the coordinates or if 
$$
|\Mult(3,B)\cup \Mult(5,B)\cup \Mult(7,B) \cup \Mult(11,B)|<5,
$$
then we are through. So suppose equality holds in all coordinates and the four sets $\Mult(j,B)$, $j=3,5,7,11$, are disjoint. Let $B_1:= \{a_2,a_{14}\}$. Note that $\Mult(3,B_1 )=\emptyset$. If $\Mult(5,B_1)=\emptyset$, then $B_1\perp \{2,3,5,7,11\}$ and since 
$|\Mult(13,B_1)|\leq 1$, some element of $B_1$ is coprime to all the other elements in $A$. If $|\Mult(5,B_1)|=1$, then $|\Mult(3,\{a_4,a_{12}\})|=|\Mult(5,\{a_4,a_{12}\})|=1$. Consequently, there is 
$x\in \{a_6,a_8,a_{10}\}$ such that $11 \mid x$,  
and hence $x\perp \{2,3,5,7\}$. It follows that $x\perp A\setminus\{x\}$. 
This completes the proof of Theorem~\ref{progressions}.

\begin{corollary}[Pillai]
\label{Consec}
In any sequence of at most $16$ consecutive integers, there exists an element that is relatively prime to all the others.
\end{corollary}
\begin{proof}
This is just the case $d=1$ of Theorem~\ref{progressions}.
\end{proof}

Recall that an integer is said to be a \emph{perfect power} if it is of the form $t^r$ where $t$ and $r$ are integers $> 1$. 
 
\begin{corollary}
Suppose  $a,d$ are coprime positive integers and $n$ is a positive integer $\le 16$ such that no term of $\AP(a,d,n)$ is a perfect power. 
Then the product $\prod_{k=0}^{n-1}{\left(a+kd\right)}$ is not a perfect power.
\end{corollary}  

\begin{proof}
The case $n=1$ is trivial. If $n \geq 2$, then by Theorem \ref{progressions}, there is a term $x$ in $\AP(a,d,n)$ that is coprime to the other terms. If $x>1$, then the desired result is clear since $x$ is not a perfect power. In case $x=1$, we must have $a=1$ and so one can apply Theorem \ref{progressions} to $\AP(a+d, d, n-1)$. 
\end{proof}

\begin{remark}
The hypothesis above that no term of $\AP(a,d,n)$ is a perfect power is crucial since one can find infinitely many arithmetic progressions (with the first term coprime to the common difference) with $3$ terms each of which is a square. This follows from the fact that  there are infinitely many rational points on the curve $x^2+y^2=2$. For instance, $\{1,5^2,7^2\},\{7^2,13^2,17^2\}$ and $\{17^2,53^2,73^2\}$.   
\end{remark}

\begin{remark}
A remarkable theorem of Erd{\H{o}}s and Selfridge~\cite{ES} says that the product of two or more consecutive positive integers is never a perfect power.
\end{remark}

\begin{remark}
\label{evans}
As mentioned in the Introduction,
for each $n>16$ there exist blocks of $n$ consecutive integers such that they contain no integer relatively prime to all the rest. We refer to Evans~\cite{Ev1} for an elegant proof of this result.
\end{remark}

\section{GCD domains and Decomposition Numbers}
\label{general}

The notion of a GCD domain was recalled in the Introduction. 
Let us also recall that an integral domain 
$R$ is said to be a \emph{B\'{e}zout  domain} if every finitely generated ideal of $R$ is principal, 
while it is said to be \emph{atomic} if 
every nonzero nonunit in $R$ factors into a product of irreducible elements. In analogy with the latter, we shall say that an integral domain 
$R$ is \emph{$\sigma$-atomic} if every nonzero nonunit in $R$ is divisible by an irreducible element. Evidently, a unique factorization domain (UFD)
is a $\sigma$-atomic (in fact, atomic) GCD domain. The following example shows that the converse is not true.

\begin{example}
\label{entire}
Let $R$ be the ring of entire functions, i.e., complex-valued holomorphic functions on $\C$. The units of $R$ are precisely the entire functions with no zeros in $\C$ and the irreducible elements are given, up to multiplication by units, by the linear polynomials. 
Thus it is readily seen that $R$ is a $\sigma$-atomic domain. Moreover, $R$ is also a GCD domain, and in fact, a B\'{e}zout domain, thanks to a result of Helmer \cite{Hel} (which was, incidentally, published in the same year as Pillai \cite{P1}). On the other hand, since there do exist entire functions with infinitely many zeros (e.g., $\sin z$), we see that $R$ is not a UFD. More generally, 
if $R'$ is any subring of $R$ such that $R'$ strictly contains 
the subring of $R$ consisting of  the polynomial functions,  
then $R'$ is a $\sigma$-atomic GCD domain that is not a UFD. To generate more examples, 
it suffices to observe that if $S$ is a $\sigma$-atomic GCD domain that is not a UFD, then the polynomial ring 
$S[X]$ is also a $\sigma$-atomic GCD domain that is not a UFD; moreover, it is not difficult to see that $S[X]$ is neither Noetherian 
nor B\'{e}zout. 
\end{example}

In the sequel, the following version of Chinese Remainder Theorem will turn out to be useful. A proof when $R=\Z$ can be found in the book of
Ore \cite[\S 10--3]{ore} 
and it 
extends easily to the case when $R$ is any B\'{e}zout domain, or more generally, a GCD domain where the 
moduli satisfy a B\'{e}zout hypothesis such as \eqref{Bez} below. 

\begin{lemma}[Generalized Chinese Remainder Theorem]
\label{cong}
Let $R$ be a GCD domain, $m$ be a positive integer, and let $u_i,v_i \in R$ with $v_i \neq 0$  for $i=1, \dots , m$. Assume that
\begin{equation}
\label{Bez} 
\gcd{(v_i,v_j)}\in Rv_i+Rv_j \quad \text{ for  } 1\leq i,j \leq m.
\end{equation}
Then the system 
$z \equiv u_i\imod{v_i}$, $i=1, \dots , m$, 
of $m$ congruences  possesses a solution  in $R$ if and only if $\, \gcd(v_i,v_j) \mid (u_i-u_j)$  for  $1\leq i,j \leq m$.
\end{lemma} 

\begin{remark}
\label{ConsecAndAP}
As an application of Lemma \ref{cong}, let us show that the Pillai Theorem (Corollary~\ref{Consec}) and 
the Generalized Pillai Theorem (Theorem~\ref{progressions}) can be deduced from each other. 
To prove the nontrivial implication, let $a,d$ be coprime integers and $n$ be a positive integer $\le 16$. Write $a_i=a+(i-1)d$ for $i=1,\dots ,n$. By Lemma~\ref{le1}, $\gcd(a_i,a_j)\mid i-j$ for $1\le i,j\le n$. Hence by Lemma \ref{cong}, there is $z\in \Z$ such that $a_i \mid z-i\, $ for $i=1,\dots ,n$. 
Now by  Corollary~\ref{Consec}, there is $k\in \{1, \dots ,n\}$ such that $z-k$ is relatively prime to $z-j$ for all $j=1, \dots ,n$ with $j\ne k$. Consequently, $a_k$ is relatively prime to $a_j$ for all $j\ne k$.
\end{remark}

It is not difficult to show that in a GCD domain, irreducible elements are always prime. In what follows, prime elements of an arbitrary integral domain may simply be referred to as primes, and this terminology should not be confused with prime ideals.  
Also, following the standard conventions of number theory, we will use the term rational prime to mean a (positive) prime number in $\Z$. 
Now here is a definition that will play a crucial role in the proof of our main theorem.

\begin{definition}
 Let $\RR$ be an integral domain with multiplicative identity $1_\RR$. The \emph{decomposition number} of $\RR$, denoted by ${\delta}_\RR$, is 
 the smallest 
rational prime $p$ such that $p \cdot 1_\RR$ is divisible by at least two distinct (i.e., up to multiplication by units) prime elements in $\RR$. If no prime in $\Z$ is divisible by two distinct primes in $\RR$, then we define ${\delta}_\RR$ to be $\infty$.  
\end{definition}

It seems worthwhile to illustrate this notion with several examples. 

\begin{examples}
(i) Clearly,  ${\delta}_\Z=\infty$. Also, if $K$ is a field, then ${\delta}_K=\infty$.
\smallskip 

(ii) If $A$ is an integral domain, then ${\delta}_{A[[X]]}={\delta}_{A[X]}={\delta}_A$. 
\smallskip
 
(iii) If $R$ is the ring of entire functions, then ${\delta}_R=\infty$. 
%
\smallskip
 
(iv) Suppose $\RR=\Z[\alpha]$ is a UFD for some complex number $\alpha$ satisfying a monic irreducible polynomial $f(X)\in \Z[X]$. Then using a well-known result of Kummer-Dedekind, we see that ${\delta}_\RR$ is the smallest 
rational prime $p$ such that the image of $f(X)$ in $\Z/p\Z[X]$ is divisible by two distinct irreducible polynomials in $\Z/p\Z[X]$. The next two examples are special cases of this. 
\smallskip
 
(v) Let $K=\Q(\sqrt{m})$ for some squarefree $m\in \Z$ such that $\RR=\Ok$ is a UFD. If $m\equiv 1$(mod 8), then ${\delta}_\RR=2$ and for 
other values of $m$, ${\delta}_\RR$ is the smallest rational prime $p$ 
such that $p$ is odd and $\left(\frac{m}{p}\right)=1$. (Here $\left(\frac{\cdot}{\cdot}\right)$ denotes the \emph{Legendre symbol}). In particular, 
 ${\delta}_{\Z[i]}=5$. 
\smallskip

(vi) Suppose $\zeta_m$ is a primitive $m$-th root of unity such that $\RR=\Z[\zeta_m]$ is a UFD. (The precise values of $m$ for which $\Z[\zeta_m]$ is a UFD are known; cf.~\cite{MaM}). For any rational prime $p$, let $p^{v_p(m)}$ be the highest power of $p$ dividing $m$ and let $\ell(m,p):= m/p^{v_p(m)}$. 
Then from \cite[Theorem 2]{Bach}, 
we see that ${\delta}_\RR$ is the smallest rational prime $p$ 
for which $\ell(m,p)>1$ and $p$ is not a primitive root modulo ${\ell(m,p)}$.
\end{examples}

\section{Arithmetic Progressions in GCD domains}
\label{last}

For an integral domain $\RR$, we shall denote by $\Z_{\RR}$ the prime subring of $\RR$. In case $\RR$ is of characteristic zero, $\Z_{\RR}$ can be identified with $\Z$. 

\begin{theorem}
\label{gen}
Let $\RR$ be a $\sigma$-atomic GCD domain of characteristic $0$. 
Then we have the following:
\begin{enumerate} 
\item[{\rm (i)}] 
If $n$ is a positive integer 
$\le \min{\{16,1+{\delta}_\RR\}}$, then for any coprime $a,d \in \RR$,  the arithmetic progression $\AP(a,d,n)$ contains a term that is relatively prime to all the others. 
\item[{\rm (ii)}]
Assume that no prime of $\Z_{\RR}$ 
is a unit in $\RR$. Then for each integer $n$ with $n>\min{\{16,1+{\delta}_\RR\}}$, there exists an arithmetic progression $\AP(a,d,n)$ in $\RR$,  
where $a,d \in \RR$ are coprime, such that none of its terms is relatively prime to all the others.
\end{enumerate}
\end{theorem}

\begin{proof}
(i) Assume, on the contrary, that $n \leq \min{\{16,1+{\delta}_\RR\}}$ and that there exists an arithmetic progression $a_1,\ldots,a_n$ in $\RR$ with $\gcd{(a_1,a_2)}=1$ such that no term is relatively prime to the rest. Let $d=a_2-a_1$. For each $i\in\{1, \dots , n\}$, there exists a $j_i\in\{1, \dots , n\}$ such that $a_i$ is not coprime to $a_{j_i}$. Then there exists a prime $P_i$ dividing $\gcd{(a_i,a_{j_i})}$. Let $\pi$ be the product of the distinct primes in $\{P_1,\ldots,P_n\}$, say
$$
\pi=P_{i_1} P_{i_2} \cdots P_{i_k}.
$$
Let $r_i=\gcd{(\pi,a_i)}$. Note that no $r_i$ is relatively prime to all the $r_j (j\neq i)$. For any prime $P\in \RR$ let $I_P=\{i:P\mid r_i\}$. Suppose 
$$
\bigcup_{j=1}^{k}\{I_{P_{i_j}}\}=\{I_{Q_1},I_{Q_2},\cdots,I_{Q_l}\}.
$$
where each $Q_m$ is one of the $P_{i_j}$. Now define
$$
s_i=\gcd{(r_i,Q_1 \cdots Q_l)} \quad (1\leq i \leq n).
$$
Again note that no $s_i$ is relatively prime to all the $s_j (j\neq i)$. By the choice of the $s_i$, it follows that if $Q_s \neq Q_t$, then 
$$
\left|\{i:Q_s \mid s_i\}\right|\geq 2 \quad \text{ and } \quad 
\{i:Q_s \mid s_i\} \neq \{i:Q_t \mid s_i\}.
$$
Note that $a_i \equiv 0 \imod{s_i}$. If $P$ is a prime dividing some $s_l$, then $P$ also divides some $s_m (l \neq m)$. Then $P \mid (l-m)1_\RR$. If $p\in \Z$ is the positive prime such that $P\RR\cap \Z_{\RR}=p\Z_{\RR}$, then it follows that $p \mid (l-m)$. Thus $p \leq n-1 \leq {\delta}_\RR$. This implies that each $s_i (1 \leq i \leq n)$ is a product of primes each of which lies over $p \cdot 1_{\RR}$ for some positive prime $p\in \Z$ not exceeding ${\delta}_\RR$. Also, any prime $P$ lying over ${\delta}_\RR \cdot 1_{\RR}$ can appear in the factorization of some $s_m$ only if $n={\delta}_\RR+1$ and in this case $P$ necessarily divides both $s_1$ and $s_n$ and no other $s_i$. Also, in this case no other prime $Q \in \RR$ lying over ${\delta}_\RR \cdot 1_{\RR}$ can divide $s_1$ or $s_n$ (since $\{i:P \mid s_i\} \neq \{i:Q \mid s_i\}$). Now, the system
$$
z  \equiv -(i-1)d  \imod{s_i}, \quad 1\leq i \leq n, 
$$
possesses a solution, namely $z=a_1$ in $\RR$. This implies that 
$$
\gcd{(s_i,s_j)} \mid (i-j)d \cdot 1_{\RR}\quad \text{for} \quad 1 \leq i,j \leq n.
$$
 Also note that $\gcd{(s_i,d)}=1$ for all $i$ (otherwise the fact that the first term is coprime to the common difference is contradicted). Thus we have 
\begin{equation}
\label{eq1}
\gcd{(s_i,s_j)} \mid (i-j)\cdot 1_{\RR} \quad \text{for} \quad 1\leq i,j \leq n.
\end{equation}
  For each $i\in\{1, \dots , n\}$ let $t_i \in \Z$ denote the unique positive integer such that $s_i \RR \cap \Z_{\RR} = t_i\Z_{\RR}$. Further, for $j\in\{1, \dots , n\}$, let $t_{ij}=\gcd{(t_i,t_j)}$ and $s_{ij}=\gcd{(s_i,s_j)}$. Then $t_{ij}\in t_i\Z_{\RR}+t_j\Z_{\RR}$ and consequently
\begin{equation}
 \label{eq2}
s_{ij}\in \RR s_i+\RR s_j.
\end{equation}
From \eqref{eq1} and \eqref{eq2} together with Lemma~\ref{cong}, it follows that the system
$$
z  \equiv  (1-i)1_{\RR} \imod{s_i}, \quad i = 1, \dots ,  n, 
$$
possesses a solution in $\RR$. By the choice of $s_i$ and $t_i$ it is easily seen that 
$$
s_{ij}\RR \cap \Z_{\RR} = t_{ij} \Z_{\RR}.
$$
Consequently 
$$
\gcd{(t_i,t_j)} \mid (i-j) \quad \mbox{for} \quad 1\leq i,j \leq n
$$
and so by Lemma~\ref{cong}, the system
$$
z  \equiv  1-i \imod{t_i},   \quad i = 1, \dots ,  n,  
$$
possesses a solution, say $x$, in $\Z$. Note that no $t_i$ is relatively prime to all the others. 
Now $x,x+1,\ldots,x+n-1$ is a sequence of $n$ consecutive integers such that none of them is relatively prime to all the rest. 
Hence by Corollary~\ref{Consec}, $n>16$, which is a contradiction. This proves (i). 

(ii) If $n>16$, then by Remark~\ref{evans} there exists an arithmetic progression $\AP(a,d,n)$ in $\Z$ (and hence in $\Z_{\RR}$) 
such that none of its terms is relatively prime to all the others. 
Since no prime in $\Z_{\RR}$ is a unit in $\RR$ we are through. So we only need to consider progressions where $1+{\delta}_\RR<n<17$. For ${\delta}_\RR<17$ let $P,Q$ be distinct primes in $\RR$ dividing ${\delta}_\RR 1_{\RR}$. Let $z \in \RR$ be a solution of the system
\begin{eqnarray*}
 z & \equiv & 0  \left(\text{mod} \ \, 2 \cdot 3 \cdot 5 \cdot 7 \cdot 11 \cdot 13 \cdot \frac{P}{{\delta}_\RR} \right) \\
z+1_\RR & \equiv &  0  \imod{Q}.  \\ 
\end{eqnarray*}
Then for $1+{\delta}_\RR<n<17$, we claim that the progression 
$z, \, z+1_\RR,\, \ldots, \, z+(n-1)1_\RR$  
contains no term relatively prime to all the others. 
Indeed, $z$ and $z+1_\RR$ share a common factor with $z+{\delta}_\RR 1_\RR$ and $z+({\delta}_\RR+1)1_\RR$ respectively, whereas all other terms in the progression share a common factor with either $z$ or $z+1_\RR$. This completes the proof of Theorem~\ref{gen}.
\end{proof}

\begin{remark}
In part (i) of Theorem~\ref{gen}, the hypothesis that $\RR$ is $\sigma$-atomic is crucial. For example, if $\RR$ is the ring $\mathbb{A}$ of all algebraic integers (all complex numbers integral over $\Z$), then ${\delta}_{\mathbb{A}}=\infty$ since there are no prime elements in $\mathbb{A}$. However, the progression
$$
\frac{\sqrt{17}+3}{2}, \; \frac{\sqrt{17}+5}{2}, \; \frac{\sqrt{17}+7}{2}, \; \frac{\sqrt{17}+9}{2}
$$
contains no element coprime to the others. 
Also, in part (ii) of Theorem~\ref{gen}, the assumption that no prime in $\Z$ is a unit in $\RR$ is necessary. 
For example, if $\RR$ is a field, then every arithmetic progression in $\RR$ with at least one nonzero term contains a term that 
is relatively prime to all the others.
\end{remark}

\begin{remark}
 Note that only values of ${\delta}_\RR\leq 13$ can affect the permissible values of $n$ such that $\AP(a,d,n)$ contains a term coprime to the others. For the purposes of the above theorem it is unnecessary to determine ${\delta}_\RR$ if it is known to be greater than $13$ (this of course includes the case ${\delta}_\RR=\infty$).
\end{remark}

%

\begin{definition}
 An element of an integral domain $R$ is said to be a \emph{perfect power} if it can be expressed in the form $t^r$ where $t$ is a nonzero, non-unit in $R$ and $r$ is a positive integer $> 1$.  
\end{definition}
%

\begin{corollary}
 Let $\RR$ be a UFD of characteristic 0 and let $a,d\in \RR$ be coprime. 
 If $n$ is a positive integer $\leq \min{\{16,1+{\delta}_\RR\}}$ and if $\AP(a,d,n)$ contains no units and no perfect powers, 
 then the product 
$
\prod_{k=0}^{n-1}{\left(a+kd\right)}$ is not a perfect power.
\end{corollary}

\begin{proof}
If some term of $\AP(a,d,n)$ is zero, then so is the product and $0$ is not a perfect power, by definition. Otherwise, the 
result follows readily from Theorem~\ref{gen}.
\end{proof}

\end{document}